\documentclass[12 pt]{amsart}
\usepackage{amssymb, amsmath, amsthm}
\usepackage{enumitem}
\usepackage{hyperref}
\usepackage{bbm}
\usepackage{float}
\usepackage{tikz-cd}

\DeclareMathOperator{\im}{im}
\DeclareMathOperator{\Pic}{Pic}
\DeclareMathOperator{\Spec}{Spec}
\DeclareMathOperator{\CH}{CH}
\DeclareMathOperator{\fieldchar}{char}
\DeclareMathOperator{\id}{id}

\newcommand{\mg}[1]{\mathcal{M}_{#1}}
\newcommand{\mgbar}[1]{\overline{\mathcal{M}}_{#1}}
\newcommand{\mgtilde}[1]{\widetilde{\mathcal{M}}_{#1}}

\newtheorem{theorem}{Theorem}[section]
\newtheorem{lemma}[theorem]{Lemma}
\newtheorem{proposition}[theorem]{Proposition}
\newtheorem{corollary}[theorem]{Corollary}

\theoremstyle{definition}
\newtheorem{definition}[theorem]{Definition}
\newtheorem{note}[theorem]{Note}
\newtheorem{observation}[theorem]{Observation}

\usepackage{multirow}

\title{The integral Chow ring of $\mg{1,n}$ for $n=3,\dots,10$}
\author{Martin Bishop}
\date{19 November 2023}

\begin{document}
\maketitle

\begin{abstract}
We compute the integral Chow ring of the moduli stack of smooth elliptic curves with $n$ marked points
for $3\leq n\leq 10$.
\end{abstract}

\section{Introduction}
\subsection{Contents}
The main result of this paper is the following:

\begin{theorem}\label{main result}
	Let $\lambda_1$ be the first Chern class of the Hodge bundle. Then over a field
	of characteristic not equal to 2 or 3:
	\begin{enumerate}[label=(\alph*)]
		\item $\CH(\mg{1,3})=\mathbb Z[\lambda_1]/(12\lambda_1,6\lambda_1^2)$\label{main result a}
		\item $\CH(\mg{1,4})=\mathbb Z[\lambda_1]/(12\lambda_1,2\lambda_1^2)$
		\item $\CH(\mg{1,n})=\mathbb Z[\lambda_1]/(12\lambda_1,\lambda_1^2)$, for $n=5,\dots,10$.
	\end{enumerate}
\end{theorem}

We open by reviewing some essential background:
the Weierstrass form for elliptic curves, the Chow rings of $\mg{1,1}$ and $\mg{1,2}$,
and higher Chow groups with $\ell$-adic coefficients.
We then compute the integral Chow rings of $\mg{1,n}$ for $3\leq n\leq 10$ over
a (not-necessarily algebraically closed) field $k$ with $\fieldchar k\neq 2,3$
by using higher Chow groups with $\ell$-adic coefficients in the base case
$n=3$, and then leveraging this information for larger $n$. This extends
Belorousski's computation of the \textit{rational} Chow ring of these stacks \cite{Bel98}. Along the way,
we also prove the rationality of $\mg{1,n}$ for $3\leq n\leq 10$, which was previously
only known in the case $\mathbbm k=\bar{\mathbbm k}$, $\fieldchar\mathbbm k=0$,
and analyze the notion
of the \textit{integral tautological ring} of $\mg{1,n}$.

\subsection{History}

In \cite{Mum83}, Mumford introduced the study of the intersection theory of the
coarse moduli space of genus $g$
curves, $\overline M_g$. This space is singular and its Chow ring cannot be defined with integer coefficients,
but the singularities are mild enough that it can be defined with
rational coefficients (the \textit{rational Chow ring}). Extending this notion, the rational Chow rings
of the
moduli stacks of genus $g$ stable (resp. smooth) $n$-pointed curves,
denoted $\mgbar{g,n}$ (resp.
$\mg{g,n}$), have been computed for many $(g,n)$ \cite{Bel98, CL21, Fab90a, Fab90b, Iza95, Mum83, PV15}.

However, using rational coefficients eliminates all torsion, and so ignores a rich part of the structure of
the space. Enabled by the extension of the definition of integral Chow rings to quotient stacks
by Totaro \cite{Tot99} and Edidin-Graham \cite{EG98}, Vistoli and Edidin-Graham
computed the integral Chow rings of $\mg{2}$ \cite{Vis98}, $\mg{1,1}$, and $\mgbar{1,1}$
\cite{EG98}. Then progress froze until the recent development
of new techniques for computing with integral coefficients, such as the patching lemma of
\cite{DLV21} and the \textit{higher Chow groups with $\ell$-adic coefficients} of \cite{Lar21}.
See the below table for a list of currently known values.

\begin{figure}[H]
\begin{center}
\renewcommand{\arraystretch}{1.3}
\begin{tabular}{| c | c | c |}
\hline
genus & moduli & reference\\
\hline
\multirow{2}{*}{$g=0$}
	& $\mg{0,n}$, $n\geq3$& classical\\
	\cline{2-3}
	& $\mgbar{0,n}$, $n\geq3$ & \cite{Kee92}\\
	\hline
\multirow{5}{*}{$g=1$}
	& $\mg{1,1}$ &
		\multirow{2}{*}{\cite{EG98}}\\
	\cline{2-2}
	& $\mgbar{1,1}$ &\\
	\cline{2-3}
	& $\mg{1,2}$ & \cite{Inc21}\\
	\cline{2-3}
	& $\mgbar{1,2}$ & \cite{DLPV21, Inc21}\\
	\cline{2-3}
	& $\mg{1,n}$, $3\leq n\leq10$ & [--]\\
\hline
\multirow{3}{*}{$g=2$}
	& $\mg{2}$ & \cite{Vis98}\\
	\cline{2-3}
	& $\mgbar{2}$ & \cite{Lar21}\\
	\cline{2-3}
	& $\mgbar{2,1}$ & \cite{DLV21}\\
\hline
\end{tabular}
\caption{All currently known integral Chow rings of $\mg{g,n}$ and $\mgbar{g,n}$.}
\end{center}
\end{figure}

\subsection{The patching problem}
One powerful tool for computing Chow rings is the \textit{excision exact sequence}. Given a closed
substack $p:Z\rightarrow X$ with complement $U$, there is an exact sequence
$$
\CH(Z)\xrightarrow{p_*}\CH(X)\rightarrow\CH(U)\rightarrow0.
$$
This sequence allows one to compute the Chow ring of an open locus when the Chow ring of its complement
and of the whole space are known. However, we frequently find ourselves in the opposite
situation: when dealing with complicated objects stratified by simpler ones, we may able to compute
the Chow rings of $Z$ and its complement $U$, and need to patch these together to get the Chow
ring of $X$.

This may be referred to as the \textit{patching problem}, and solving it is the crux of many Chow computations.
The above-mentioned new techniques, the patching lemma and higher Chow groups with $\ell$-adic
coefficients, give methods for solving the patching problem and have fueled
the recent explosion in progress in computing integral Chow rings.

\subsection{Conventions}
For the remainder of this paper, all schemes and stacks are over a fixed field $\mathbbm k$ of
characteristic not equal to 2 or 3.

\subsection{Acknowledgements}
I would like to thank Jarod Alper, Catherine Babecki, Andrea Di Lorenzo, Kristine Hampton,
Giovanni Inchiostro, Eric Larson,
Max Lieblich, Jessie Loucks-Tavitas, Brian Nugent, Juan Salinas, and Alex Scheffelin for the countless
helpful conversations which occurred during this project.

\section{The $\mg{1,1}$ and $\mg{1,2}$ cases}
Our analysis of $\mg{1,n}$ for higher $n$ depends in multiple places on $\mg{1,1}$
and $\mg{1,2}$, so we first review their structure, which
is essentially a corollary of the Weierstrass form for elliptic curves.
The Chow ring of $\mg{1,1}$ was first computed in
\cite{EG98} and $\mg{1,2}$ in \cite{Inc21}.

\subsection{The Weierstrass form}
We open with the classically known Weierstrass form for elliptic curves.

\begin{theorem}[Weierstrass]
Any one-pointed smooth elliptic curve over a field $\mathbbm k$ of characteristic not equal to 2 or 3
can be written in the form $y^2z=x^3+axz^2+bz^3$,
where the marked point is the point at infinity $[0:1:0]$. Moreover, if we denote such a curve by
$C_{(a,b)}$, then
$$
C_{(a,b)}\cong C_{(a',b')}\quad\text{if and only if}\quad(a',b')=(t^{-4}a,t^{-6}b).
$$
The isomorphism
between these curves is given by
$$
[x:y:z]\mapsto[t^{-2}x:t^{-3}y:z].
$$
An elliptic curve
is smooth if and only if $D=4a^3+27b^2=0$, nodal if and only if $D=0$ and $(a,b)\neq(0,0)$, and
cuspidal if and only if $(a,b)=(0,0)$. Lastly, we have
$$
H^0(\omega_C)=\left<\frac{dx}y\right>.
$$
\end{theorem}

Rephrasing this gives the following corollaries:

\begin{corollary}\label{m11 presentations}
The Weierstrass form give an isomorphism
$$
\mg{1,1}\cong\left[\frac{\mathbb A^2\setminus V(D)}{\mathbb G_m}\right],
$$
where the $\mathbb G_m$ action has weight $(-4,-6)$ and $D=4a^3+27b^2$.
\end{corollary}

\begin{corollary}\label{m12 presentation}
We have that $\mg{1,2}$ is isomorphic to an open substack of a vector bundle over $B\mathbb G_m$.
\end{corollary}

\begin{proof}
From the Weierstrass form,
a two-pointed smooth elliptic curve is determined, up to scaling, by a choice of $(a,b)$ and $(x,y)$ such that
$$
y^2=x^3+ax+b \quad\text{and}\quad D\neq0.
$$
We can solve for $b$ to see that $a,x,y$ vary freely, provided that $D\neq0$, where
$$
D=4a^3+27b^2=4a^3+27(y^2-(x^3+ax))^2.
$$
Since $\mathbb G_m$ acts with weights $-4,-2,-3$ on $a,x,y$, we conclude that $\mg{1,2}$ is open
in $\left[\frac{\mathbb A^3_{a,x,y}}{\mathbb G_m}\right]$,
where $\mathbb G_m$ acts with the above weights.
\end{proof}

\begin{corollary}\label{m11 and m12 quotients}
The rings $\CH(\mg{1,1})$ and $\CH(\mg{1,2})$ are both quotients of $\mathbb Z[x]/(12x)$.
\end{corollary}

\begin{proof}
This follows from Corollaries \ref{m11 presentations} and \ref{m12 presentation}, along with the fact
that $D$ has weight 12 under the $\mathbb G_m$ action.
\end{proof}

\begin{corollary}\label{Hodge bundle generator}
The generator of the Chow ring of $\mg{1,1}$ and $\mg{1,2}$ is $\lambda_1$, the first Chern class of the
Hodge bundle.
\end{corollary}

\begin{proof}
Since $\mathbb G_m$ acts with weights $-2$ and $-3$ on $x$ and $y$, respectively, we see that
$\frac{dx}y$ has weight $1$ under the $\mathbb G_m$ action. Hence under the pullback map
$\CH(B\mathbb G_m)\rightarrow\CH(\mgtilde{1,1})$ we have $x\mapsto\lambda_1$.
Since by the previous Corollary $\CH(\mg{1,1})$ and $\CH(\mg{1,2})$ are generated by the pullback of $x$,
we see that they are generated by $\lambda_1$.
\end{proof}

\begin{corollary}\label{Hodge bundle pullback}
The pullback of $x\in\CH(B\mathbb G_m)$ to any moduli stack of pointed elliptic curves is $\lambda_1$.
\end{corollary}

\begin{proof}
This follows since $x$ pulls-back to $\lambda_1$ in $\CH(B\mathbb G_m)\rightarrow\CH(\mgtilde{1,1})$ and the
Hodge bundle pulls-back to the Hodge bundle.
\end{proof}

\begin{theorem}\label{n=1 Chow rings}
Let $\lambda_1$ be the first Chern class of the Hodge bundle.
Over a field of characteristic not equal to 2 or 3
\begin{enumerate}[label=(\alph*)]
\item $\CH(\mg{1,1})\cong\mathbb Z[\lambda_1]/(12\lambda_1)$
\item $\CH(\mg{1,2})\cong\mathbb Z[\lambda_1]/(12\lambda_1)$.
\end{enumerate}
\end{theorem}

\begin{proof}
From Corollaries \ref{m11 and m12 quotients} and \ref{Hodge bundle generator}
we know that $\CH(\mg{1,1})$ and $\CH(\mg{1,2})$
are both quotients of $\mathbb Z[\lambda_1]/(12\lambda_1)$.
Now consider any two-pointed elliptic curve with $\mu_3$ automorphisms, such as\\
$(C_{(0,1)},\infty, [0:1:1])$, and with $\mu_4$ automorphisms, such as
$(C_{(1,0)},\infty, [0:0:1])$. These induce residual gerbes
$$
\begin{tikzcd}
B\mu_n \arrow[dr]\arrow[rr]& & \mg{1,2}\arrow[dl]\\
& B\mathbb G_m&
\end{tikzcd}
$$
for $n=3,4$. On Chow rings this induces
$$
\begin{tikzcd}
B\mu_n& & \mg{1,2}\arrow[ll]\\
& B\mathbb G_m \arrow[ur] \arrow[ul]&
\end{tikzcd}
$$

Since $\CH(B\mathbb G_m)\rightarrow\CH(B\mu_n)$ is surjective, we see that
$\CH(\mg{1,2})$ surjects onto $\mathbb Z[x]/(nx)$ for $n=3,4$. Therefore $\CH(\mg{1,2})=
\mathbb Z[\lambda_1]/(12\lambda_1)$. Considering these curves as one-pointed elliptic curves shows that
$\CH(\mg{1,1})\cong\mathbb Z[\lambda_1]/(12\lambda_1)$ as well.
\end{proof}

\section{Aside: Higher Chow groups with $\ell$-adic coefficients}
In \cite{Bl86}, Bloch introduced higher Chow groups, which complete the excision exact
sequence into a long exact sequence. They are defined as the homology of a certain complex
named $z^*(X,\bullet)$ and are, unfortunately, usually rather difficult to compute.
In \cite{Lar21}, Larson used higher Chow groups with $\ell$-adic coefficients to remedy this. Without getting
into too much detail, we list here some important properties that higher Chow groups with $\ell$-adic
coefficients possess, along with some important computations.

\begin{definition}
Define the $n^{\text{th}}$ higher Chow group with $\ell$-adic coefficients to be
$$
\CH(X,n;\mathbb Z_{\ell})=H_n\left(\lim z^*(X_{\bar k},\bullet)\otimes^L\mathbb Z/\ell^m\mathbb Z\right).
$$
\end{definition}
In the case where each $\CH(X,n;\mathbb Z/\ell^m\mathbb Z):=H_n(z^*(X_{\bar k},\bullet)\otimes\mathbb Z/\ell^m\mathbb Z)$
is finitely generated, we have
$$
\CH(X,n;\mathbb Z_{\ell})=\lim\CH(X_{\bar k},n;\mathbb Z/\ell^m\mathbb Z).
$$

\begin{proposition}
If $Z\rightarrow X$ is closed with complement $U$ and:
\begin{itemize}
\item $\CH(Z)$ and $\CH(U)$ are finitely generated,
\item $\CH(Z)\rightarrow\CH(Z_{\bar k})$ is injective,
\item there exists at least one $\ell$ for which $\CH(U,1;\mathbb Z_{\ell})=0$,
\item and $\CH(U,1;\mathbb Z_l)=0$ whenever $\CH(Z)$ has $\ell$-torsion,
\end{itemize}
then the excision sequence is exact on the left.
\end{proposition}

\begin{proof}
Notice that, at first glance, $\ell$-adic higher Chow groups tell us about the injectivity of
the excision sequence with all spaces base-changed to $\bar k$. However, we can infer
the injectivity of $\CH(Z)\rightarrow\CH(X)$ via the following diagram
$$
\begin{tikzcd}
\CH(Z)\arrow[r,]\arrow[d, hookrightarrow] & \CH(X)\arrow[d]\\
\CH(Z_{\bar k})\arrow[r] & \CH(X_{\bar k})
\end{tikzcd}
$$
Let $\alpha\in\CH(Z)$, and, abusing notation, refer to its image in $\CH(Z_{\bar k})$ as $\alpha$ as well.
Pick an $\ell$ such that $\alpha$ is $\ell$-torsion (if $\alpha$ is not torsion, then pick any
$\ell$ where $U$'s first $\ell$-adic higher Chow group vanishes). This gives
$$
\CH(Z)\otimes\mathbb Z_{\ell}\hookrightarrow
\CH(Z_{\bar k})\otimes\mathbb Z_{\ell}\hookrightarrow\CH(X_{\bar k})\otimes\mathbb Z_{\ell},
$$
and so the image of $\alpha$ under $\CH(Z)\rightarrow\CH(X)$ cannot vanish.
\end{proof}

\begin{proposition} Suppose $\ell$ is coprime to $\fieldchar k$ (later we will always have $\ell=2$ or $3$).
Then
\begin{enumerate}[label=(\alph*)]
\item $\CH(\Spec k,1;\mathbb Z_l)=0$.
\item $\CH(\mathbb A^n,1;\mathbb Z_l)=0$.
\item $\CH(\mathbb P^n,1;\mathbb Z_l)=0$.
\item $\CH(B\mathbb G_m,1;\mathbb Z_l)=0$.
\item $\CH(B\mu_n,1;\mathbb Z_l)=0$.
\end{enumerate}
\end{proposition}

\begin{proof}
As noted in \cite{Lar21}, (a) is a consequence of motivic cohomology. Then
(b) and (c) follow from the vector and projective bundle formulas, and (d) follows from computing
equivariantly. The last follows from the excision sequence
$$
0\rightarrow\CH(B\mathbb G_m)\rightarrow\CH([\mathbb A^1/\mathbb G_m])\rightarrow\CH(B\mu_n)\rightarrow0.
$$
\end{proof}

\section{The $\mg{1,n}$ case for $n=3,\dots,10$}

We will now compute the integral Chow rings of $\mg{1,n}$ for $n=3,\dots,10$. The overall structure
of the computation is to stratify $\mg{1,n}$ into an open whose complement stratifies into
closed substacks which are isomorphic to opens inside of $\mg{1,n-1}$.

\subsection{The integral Chow ring of $\mg{1,3}$}
We first stratify $\mg{1,n}$ into the open locus where $p_2\neq\iota(p_3)$ and the divisor where $p_2=\iota(p_3)$.

\begin{definition}\label{U_n definition}
For $n\geq2$:
\begin{enumerate}[label=(\alph*)]
\item Let $U_n\subseteq\mg{1,n}$ be the locus where $p_2\neq\iota(p_3)$.

\item Let $U_n'\subseteq\mg{1,n}$ be the
locus where $p_2\neq\iota(p_i)$ for any $i$.
\end{enumerate}
\end{definition}

\begin{note}
Note that the condition for $U_2$ is ill-defined. We use the convention that this is an empty condition,
so that $U_2=\mg{1,2}$.
\end{note}

\begin{observation}
Where $\pi$ is, as usual, the map $\pi:\mg{1,n}\rightarrow\mg{1,n-1}$ forgetting the last marked point,
we always have $\pi(U_{n+1})\subseteq U_n$ and
$\pi(U_{n+1}')\subseteq U_n'$, and for $n\geq 4$ we have $\pi^{-1}(U_{n-1})=U_n$.

Therefore we have induced pullback maps on Chow rings given by $\pi^*$. Since
$\pi^*(\lambda_1)=\lambda_1$, we see that
$\pi$ pulls relations back to relations: if $a\lambda_1=0$ in $\CH(U_m)$ or $\CH(U_m')$
for some $m$,
then $a\lambda_1=0$ on that same locus for all $n\geq m$.
\end{observation}

\begin{definition}\,
For $n\geq3$,
define the morphism of stacks
$\sigma_{n-1}:U_{n-1}'\rightarrow\mg{1,n}$ by
$$
(C,p_i)\mapsto(C,p_1,p_2,\iota(p_2),p_3,\dots,p_{n-1}).
$$
\end{definition}

This map sheds light on why the loci in Definition \ref{U_n definition} were defined: the defining conditions
for $U_n'$ are precisely the conditions needed to insure that this map exists.

\begin{proposition}\label{sigma closed immersion}
For $n\geq3$, the map $\sigma_{n-1}:U'_{n-1}\rightarrow\mg{1,n}$ is a closed immersion.
\end{proposition}

\begin{proof}
Let $W$ be the locus inside of $\mg{1,n}$ where $p_2\neq\iota(p_2)$.
Observe that $\sigma_{n-1}$ factors through this locus: the image of $\sigma_{n-1}$ is the locus where
$p_3=\iota(p_2)$, hence in particular $p_2\neq\iota(p_2)$.
Let $\pi_3:\mg{1,n}\rightarrow\mg{1,n-1}$ be the map which forgets the third marked point.
Then we have the following diagram:
$$
\begin{tikzcd}
W\arrow[r]\arrow[d, swap, "\pi_3"] & \mg{1,n}\\
U'_{n-1}\arrow[u, bend right, swap, "\sigma_{n-1}"]\arrow[ur, swap, "\sigma_{n-1}"] &
\end{tikzcd}
$$
where $\pi_3\circ\sigma_{n-1}=\id$. Therefore $\sigma_{n-1}:U'_{n-1}\rightarrow\mg{1,n}$
factors as a closed immersion into $W$ followed by an open immersion into $\mg{1,n}$. Since
its image in $\mg{1,n}$ is the closed locus of curves with $p_3=\iota(p_2)$, we see that
$\sigma_{n-1}:U'_{n-1}\rightarrow\mg{1,n}$ is a
closed immersion.
\end{proof}

\begin{corollary}
For $n\geq 3$, the stack $\mg{1,n}$ stratifies into the disjoint union
$$
\mg{1,n}=U_n\sqcup\im\sigma_{n-1}\cong U_n\sqcup U_{n-1}'.
$$
\end{corollary}

\begin{lemma}\label{U3 vector bundle}
The stack $U_3$ is isomorphic to an open substack of a vector bundle $\mathcal U_3$ over $B\mu_2$.
\end{lemma}

\begin{proof}
A smooth three-pointed elliptic curve
is determined, up to scaling, by a choice of $(a,b)$, $p_2=(x_2,y_2)$, and $p_3=(x_3,y_3)$ such that
$$
y_i^2=x_i^3+ax_i+b\quad\text{and}\quad{D\neq0}.
$$
Solving for $b$ and then $a$ gives
$$
a=\frac{(y_3^2-x_3^3)-(y_2^2-x_2^3)}{x_3-x_2}.
$$
Therefore we see that $x_2,x_3,y_2,y_3$ may vary freely, provided $x_2\neq x_3$ and $D\neq 0$.
But the condition that $x_2\neq x_3$ is precisely the condition that $p_2$ and $p_3$ do not
overlap and are not involutions of each other
involutions of each other (the defining condition for $U_3$), and so $U_3$ is open inside of
$$
\mathcal U_3:=\left[\frac{\mathbb A^2_{y_i}\times(\mathbb A^2_{x_i}\setminus\Delta)}{\mathbb G_m}\right],
$$
where $\Delta$ is the diagonal and $\mathbb G_m$ acts with weight $-2$ on $x_i$ and $-3$ on $y_i$.
This is a vector bundle over
$$
\left[\frac{\mathbb A^2_{x_i}\setminus\Delta}{\mathbb G_m}\right]
$$
which is a vector bundle over
$$
\left[\frac{\mathbb A^1\setminus 0}{\mathbb G_m}\right]\cong B\mu_2
$$
since $\mathbb G_m$ acts with weight $-2$ on $x_i$.
\end{proof}

\begin{lemma}\label{U2' vector bundle}
The stack $\im\sigma_2$ is isomorphic to an open substack of a vector bundle
$\mathcal U_2'$ over $B\mu_3$.
\end{lemma}

\begin{proof}
Since $\im\sigma_2$ is isomorphic to the locus $U_2'$ in $\mg{1,2}$,
we just need to analyze two-pointed elliptic curves where $\iota(p_2)\neq p_2$.
Recall from Corollary \ref{m12 presentation} that $\mg{1,2}$ is open inside of a vector bundle over
$B\mathbb G_m$. More specifically, $\mg{1,2}$ is open inside of $[\mathbb A^3/\mathbb G_m]$ with coordinates
$a,x,y$. The condition that $\iota(p_2)\neq p_2$ is equivalent to the condition
$y\neq 0$, since $\iota(p_2)=\iota([x:y:z])=[x:-y:z]$. Therefore $U_2'$ is open inside of
$$
\mathcal U'_2:=\left[\frac{\mathbb A^3\setminus\{y=0\}}{\mathbb G_m}\right]\cong\left[
\frac{\mathbb A^2_{a,x}\times(\mathbb A^1_y\setminus0)}{\mathbb G_m}
\right]
$$
which is a vector bundle over
$$
\left[\frac{\mathbb A^1_y\setminus 0}{\mathbb G_m}\right]\cong B\mu_3,
$$
since $\mathbb G_m$ acts with weight $-3$ on $y$.
\end{proof}

Now we compute the integral Chow ring of $\mg{1,3}$ by first observing
that the vector bundles $\mathcal U_3$ and $\mathcal U_2'$
of the previous section naturally live inside of an enlargement of the stack
$\mgtilde{1,3}$, where the second and third marked points are not required to be in the nodal locus.
We patch those vector bundles together inside of this stack
using higher Chow groups with $\ell$-adic coefficients, and from there deduce $\CH(\mg{1,3})$.

\begin{lemma}
Let $\mathcal X=\mathcal U_3\sqcup\mathcal U_2'$. Then
$$\CH(\mathcal X)=\frac{\mathbb Z[x]}{(6x^2)}\quad\text{and}\quad\CH(\mathcal X,1;\mathbb Z_{\ell})=0
$$
for $\ell$ coprime to $\fieldchar\mathbbm k$.
\end{lemma}

\begin{proof}
Recall that because $\mathcal U_3$ and $\mathcal U_2'$ are both quotients by $\mathbb G_m$
and vector bundles over $B\mu_2$ and $B\mu_3$, respectively, that their first higher Chow
groups with $\ell$-adic coefficients vanish for $\ell$ co-prime to $\fieldchar\mathbbm k$ and that
their Chow rings are
$$
\CH(\mathcal U_3)=\frac{\mathbb Z[x]}{(2x)}\quad\text{and}\quad\CH(\mathcal U_2')=\frac{\mathbb Z[x]}{(3x)}
$$
where in both rings $x$ denotes the pullback of the generator $x\in\CH(B\mathbb G_m)=\mathbb Z[x]$.

Consider the following diagram
$$
\begin{tikzcd}
\mathcal U_2'\arrow[r,"\sigma_2"] \arrow[dr,swap, "\pi_2"]&
\mathcal X\arrow[d,"\pi"] & \mathcal U_3\arrow[l,swap, "j"]\arrow[dl,"\pi_3"]\\
& B\mathbb G_m &
\end{tikzcd}
$$
Denote the pullback of $x\in\CH(B\mathbb G_m)$ to $\CH(\mathcal X)$ by $x$ as well, so that the pullback
of $x$ along any map is again $x$.

Since $\CH(\mathcal U_3,1;\mathbb Z_{\ell})$ vanishes for $\ell$ co-prime to $\fieldchar\mathbbm k$,
the excision sequence
for $\mathcal U_3$ and $\mathcal U_2'$ gives
$$
0\rightarrow
\CH(\mathcal U_2')\xrightarrow{\sigma_{2_*}}\CH(\mathcal X)\rightarrow\CH(\mathcal U_3)\rightarrow0
$$
$$
0\rightarrow\mathbb Z[x]/(3x)\xrightarrow{\sigma_{2_*}}\CH(\mathcal X)\rightarrow\mathbb Z[x]/(2x)\rightarrow 0.
$$
Moreover, since $\CH(\mathcal U_2',1;\mathbb Z_{\ell})=0$ for $\ell$ coprime to $\fieldchar\mathbbm k$,
we see that $\CH(\mathcal X,1;\mathbb Z_{\ell})=0$.

In all degrees $k\geq2$, the above sequence looks like
$$
0\rightarrow\mathbb Z/3\rightarrow\CH^k(\mathcal X)\rightarrow\mathbb Z/2\rightarrow0,
$$
and so $\CH^k(\mathcal X)\cong\mathbb Z/6\cong\CH^k(\mathbb Z[x]/(6x^2))$ for $k\geq2$.
In degree one the sequence looks like
$$
0\rightarrow\mathbb Z\xrightarrow{{\sigma_2}_*}\CH^1(\mathcal X)\xrightarrow{j^*}\mathbb Z/2\rightarrow0.
$$
We now have either $\CH^1(\mathcal X)\cong\mathbb Z$ or $\mathbb Z\oplus\mathbb Z/2$, and we seek to
show that $\CH^1(\mathcal X)\cong\mathbb Z$.

Since $j^*(x)=x$ in any case, if the sequence splits then $x$ has order 2 in $\CH^1(\mathcal X)$.
However the restriction of $x$ to $\mg{1,3}\subseteq\mathcal X$ is $\lambda_1$, which does not have
order 2 (in fact it has order 12, see Theorem \ref{mg1n Picard group}).

Therefore the sequence does not split, and we conclude
that
$$\CH(\mathcal X)=\frac{\mathbb Z[x]}{(6x^2)}.$$
\end{proof}

\begin{corollary}
The Chow ring of $\mg{1,3}$ is a quotient of $\mathbb Z[\lambda_1]/(6\lambda_1^2)$.
\end{corollary}

\begin{proof}
This follows from the excision sequence, since $\mg{1,3}$ is open in $\mathcal X$, namely
it is the complement of the loci of cuspidal curves, nodal curves, and curves where the marked points
coincide. The fact that it is generated by $\lambda_1$ is a consequence of Corollary \ref{Hodge bundle pullback}.
\end{proof}

To conclude precisely which quotient, we use the following Theorem.

\begin{theorem}[\cite{FV20}]\label{mg1n Picard group}
The Picard group of $\mg{1,n}$ is isomorphic to $\mathbb Z/12$ for all $n$, generated by
the Hodge bundle.
\end{theorem}

\begin{theorem}
The integral Chow ring of $\mg{1,3}$ is
$$
\CH(\mg{1,3})=\frac{\mathbb Z[\lambda_1]}{(12\lambda_1, 6\lambda_1^2)}.
$$
\end{theorem}

\begin{proof}
The inclusion of any three-pointed curve with $\mu_2$ automorphisms, such as
$(C_{(-1,0)},\infty,[1:0:1],[0:0:1])$, or $\mu_3$ automorphisms, such as
$(C_{(0,1)},\infty,[0:1:1],[0:-1:1])$,
shows that
$\CH(\mg{1,3})$ surjects onto $\mathbb Z[x]/(2x)$ and $\mathbb Z[x]/(3x)$, respectively. Since
$\Pic(\mg{1,3})=\mathbb Z/12$, generated by $\lambda_1$, the theorem is proven.
\end{proof}

\begin{note}
Our argument can in fact be modified to not use higher Chow groups, in a similar fashion
as the argument for $\mg{1,2}$. However,
this version of the argument allows us to conclude that the first higher Chow group
of the stack $\mathcal X$ with $\ell$-adic coefficients vanishes, which may prove useful in future
computations.
\end{note}

\subsection{The case $4\leq n\leq 10$}
We first make an analogous definition of the \textit{tautological ring} in the integral case.
\begin{definition}
The \textit{integral tautological ring} of $\mg{1,n}$, written $\mathcal R(\mg{1,n})$, is
the subring of the Chow ring generated by $\lambda_1$.
\end{definition}
The remainder of this section has the following structure: first, we compute the integral tautological ring
of $\mg{1,n}$ for $n\geq4$, then we show that the full Chow ring is indeed generated by $\lambda_1$
for $4\leq n\leq 10$.

\begin{corollary}\label{Un Un' tautological ring}
For $n\geq3$:
\begin{enumerate}[label=(\alph*)]
\item the integral tautological ring $\mathcal R(U_n)$ is a quotient of $\mathbb Z[\lambda_1]/(2\lambda_1)$.
\item the integral tautological ring $\mathcal R(U_n')$ is equal to $\mathbb Z$.
\end{enumerate}
\end{corollary}

\begin{proof}
We showed in Lemma \ref{U3 vector bundle} that $2\lambda_1=0$ on $U_3$,
and so this relation holds on $U_n$ for all $n\geq 3$.
We also showed in Lemma \ref{U2' vector bundle}
that $3\lambda_1=0$ on $U_2'$, and so this relation holds
on $U_3'$ and hence on $U_n'$ for all $n\geq 3$.

Since $U_n'\subseteq U_n$ for all $n$, we see that for all $n\geq3$ both relations $2\lambda_1=0$ and
$3\lambda_1=0$ hold on $U_n'$. Therefore $\lambda_1=0$ on $U_n'$ for $n\geq 3$.
\end{proof}

\begin{lemma}
For all $n\geq4$, the integral tautological ring of
$\mg{1,n}$ is a quotient of $\mathbb Z[\lambda_1]/(12\lambda_1,2\lambda_1^2)$.
\end{lemma}

\begin{proof}
Since by Corollary \ref{Un Un' tautological ring} the tautological ring of $U_n$ is a quotient of
$\mathbb Z[\lambda_1]/(2\lambda_1)$, we can write
$$
\mathcal R(U_n)=\frac{\mathbb Z[\lambda_1]/(2\lambda_1)}I
$$
for some ideal $I$.
The excision sequence for $\im\sigma_{n-1}\cong U_{n-1}'$ gives
$$
\mathcal R(U_{n-1}')\rightarrow\mathcal R(\mg{1,n})\rightarrow\mathcal R(U_n)\rightarrow0
$$
$$
\mathbb Z\rightarrow\mathcal R(\mg{1,n})\rightarrow\frac{\mathbb Z[\lambda_1]/(2\lambda_1)} I\rightarrow0.
$$
Since the image of the morphism lands in degree one and the Picard group of $\mg{1,n}$
is known to be $\mathbb Z/12$, the lemma follows.
\end{proof}

\begin{proposition}\label{mg14 tautological ring}
The integral tautological ring of $\mg{1,4}$ is
$$
\mathcal R(\mg{1,4})=\frac{\mathbb Z[\lambda_1]}{(12\lambda_1, 2\lambda_1^2)}.
$$
\end{proposition}

\begin{proof}
Observe that by Appendix \ref{autappendix} there still exists four-pointed smooth elliptic
curves with $\mu_2$-automorphisms: $n=4$ is the largest $n$ for which such a curve exists, and all
such curves have $\mu_2$-automorphisms, generated by the involution.
Moreover, such a curve is necessarily contained inside
of $U_4$, the locus where the second and third points are not involutions of each other,
since each marked point is fixed by the involution.
Therefore we get a surjection
$$
\CH(U_4)\twoheadrightarrow\mathbb Z[x]/(2x).
$$
However,
since the degree one generator of $\CH(U_4)$ is $\lambda_1$, this morphism in fact factors as
$$
\begin{tikzcd}
\mathcal R(U_4)\arrow[r, hookrightarrow]\arrow[rr, bend left, twoheadrightarrow] &
\CH(U_4)\arrow[r, twoheadrightarrow] & \mathbb Z[x]/(2x)
\end{tikzcd}
$$
Therefore $\mathcal R(U_4)=\mathbb Z[\lambda_1]/(2\lambda_1)$, and so
$\mathcal R(\mg{1,4})=\mathbb Z[\lambda_1]/(12\lambda_1, 2\lambda_1^2)$.
\end{proof}

Before we can compute the integral tautological ring for $n\geq5$, we must analyze
$\mg{1,4}$ more thoroughly.

\begin{definition}
Let $Z_n\subseteq\mg{1,n}$ be the locus of curves with non-trivial automorphisms.
\end{definition}

\begin{observation}\label{automorphism observation}
Since $\mg{1,4}\setminus Z_4$ is a four-dimensional variety, we must have $\lambda_1^5=0$ on this locus,
and hence on any locus inside of it. Moreover, observe that every curve in $Z_4$ must have
$p_2$, $p_3$, and $p_4$ collinear: the only four-pointed smooth elliptic curves with automorphisms
are the ones where $y_i=0$ for $i=2,3,4$, and hence $p_2$, $p_3$, and $p_4$ lie on the line $y=0$
(see Appendix \ref{autappendix}).
\end{observation}

We now give a second stratification of $\mg{1,n}$ for $n\geq3$ as follows:
\begin{itemize}
\item the open locus where $p_2$, $p_3$, and $p_4$ are not colinear under the Weierstrass embedding.

\item and the divisor where $p_2$, $p_3$, and $p_4$ \textit{are} colinear.
\end{itemize}

\begin{definition}\label{V_n definition} For $n\geq 3$ define the loci:
\begin{enumerate}[label=(\alph*)]
\item Let $V_n\subseteq\mg{1,n}$ be the locus where $p_2+p_3\neq\iota(p_4)$.

\item Let $V_n'\subseteq\mg{1,n}$ be the locus where $p_2+p_3\neq\iota(p_i)$ for any $i=1,\dots,n$.
\end{enumerate}
\end{definition}

\begin{observation}
As before, $\pi$ pulls relations back to relations.
\end{observation}

\begin{definition}\,
Define the morphism of stacks $\tau_{n-1}:V_{n-1}'\rightarrow\mg{1,n}$ by
$$
(C,p_i)\mapsto(C,p_1,p_2,p_3,\iota(p_2+p_3),\dots,p_{n-1}).
$$
\end{definition}

\begin{proposition}
The map $\tau_{n-1}:V_{n-1}'\rightarrow\mg{1,n}$ is a closed immersion.
\end{proposition}

\begin{proof}
Similar to Proposition \ref{sigma closed immersion}.
\end{proof}

\begin{corollary}
For $n\geq 3$, the stack $\mg{1,n}$ stratifies into the disjoint union
$$
\mg{1,n}=V_n\sqcup\im\tau_{n-1}\cong V_n\sqcup V_{n-1}'.
$$
\end{corollary}

\begin{proposition}
For $n=2,\dots,10$ the stacks $\mg{1,n}$ are rational. Moreover, for $n=4,\dots,10$ the open
in $\mg{1,n}$ which exhibits this rationality is $U_n\cap V_n$.
\end{proposition}

\begin{proof}
This was proven by Belorousski in the case where $\mathbbm k$ is algebraically closed and characteristic zero
in \cite{Bel98} by constructing a bijective morphism between $U_n\cap V_n$
and an open subset of $\mathbb P^n$. He concludes that it is an isomorphism since $\mathbb P^n$ is normal.
This proof does not work in arbitrary characteristic (for example the Frobenius morphism on $\mathbb P^1$
is a bijective
morphism between normal varieties which is not an isomorphism). However,
Belorousski's argument showing that the morphism is bijective is in fact functorial and works in families,
therefore directly establishing that the moduli stacks are isomorphic.
\end{proof}

\begin{proposition}\label{mg1n is tautological}
The Chow ring of $\mg{1,n}$ is tautological for $n=1,\dots,10$.
\end{proposition}

\begin{proof}
We have already shown this for $n=1,2,3$. For $n\geq 4$, stratify $\mg{1,n}$ into
$\mg{1,n}=(U_n\cap V_n)\sqcup\im\sigma_{n-1}\sqcup\im\tau_{n-1}$.
That is,
$\mg{1,n}$ is the union of the open locus where
$p_2$ and $p_3$ are not involutions and $p_2$, $p_3$, and $p_4$ are not colinear along with
the divisors where these conditions do hold. But $U_n\cap V_n$
is rational by the above Proposition
and hence generated in degree one, hence generated by $\lambda_1$, hence tautological.
Since $\im\sigma_{n-1}$ and $\im\tau_{n-1}$
are isomorphic to opens in $\mg{1,n-1}$, $\mg{1,n}$ is inductively built out of tautological pieces,
hence itself tautological. This breaks at $n=11$ since $U_{11}\cap V_{11}$ is not birational to an open in
$\mathbb P^{11}$ by \cite{Bel98}.
\end{proof}

\begin{theorem}
The integral Chow ring of $\mg{1,4}$ is given by
$$
\CH(\mg{1,4})=\frac{\mathbb Z[\lambda_1]}{(12\lambda_1,2\lambda_1^2)}.
$$
\end{theorem}

\begin{proof}
Since $\mg{1,4}$ is tautological by the above Proposition, we have
$$
\CH(\mg{1,4})=\mathcal R(\mg{1,4})=\frac{\mathbb Z[\lambda_1]}{(12\lambda_1,2\lambda_1^2)}
$$
by Proposition \ref{mg14 tautological ring}.
\end{proof}

\begin{proposition}\label{V4 V4' Chow ring}
The Chow ring of $V_4$ and $V_4'$ is $\mathbb Z$.
\end{proposition}

\begin{proof}
The excision sequence gives
$$
\CH(V_3')\xrightarrow{{\tau_3}_*}\CH(\mg{1,4})\rightarrow\CH(V_4)\rightarrow0
$$
$$
\mathbb Z[\lambda_1]/(2\lambda_1)\rightarrow\frac{\mathbb Z[\lambda_1]}{(12\lambda_1,2\lambda_1^2)}
\rightarrow\CH(V_4)\rightarrow 0.
$$
Since by Observation \ref{automorphism observation}
$\lambda_1^5=0$ on $V_4$ and $\lambda_1^5\neq0$ on $\mg{1,4}$,
we see that $\lambda_1^5$ must be in the image of ${\tau_3}_*$.
Hence ${\tau_3}_*(\lambda_1^4)=\lambda_1^5$ in
$\CH(\mg{1,4})$. But we also have ${\tau_3}_*(\lambda_1^4)={\tau_3}_*(\tau_3^*(\lambda_1^4))=\lambda_1^4
{\tau_3}_*(1)$. Therefore we must have ${\tau_3}_*(1)=\lambda_1$, and so $\CH(V_4)=\mathbb Z$.
Since $V_4'\subseteq V_4$, we then also have $\CH(V_4')=\mathbb Z$.
\end{proof}

\begin{corollary}
The integral tautological ring of $V_n$ and $V_n'$ is $\mathbb Z$ for all $n\geq 4$.
\end{corollary}

\begin{proof}
Since $\CH(V_4)=\CH(V_4')=\mathbb Z$, we in particular have $\lambda_1=0$ on $V_4$ and $V_4'$, hence
on $V_n$ and $V_n'$ for all $n\geq 4$. Therefore $\mathcal R(V_n)=\mathcal R(V_n')=\mathbb Z$.
\end{proof}

\begin{proposition}\label{mg1n tautological ring}
For $n\geq5$, the integral tautological ring of $\mg{1,n}$ is
$$
\mathcal R(\mg{1,n})=\frac{\mathbb Z[\lambda_1]}{(12\lambda_1,\lambda_1^2)}.
$$
\end{proposition}

\begin{proof}
The excision sequence gives
$$
\mathcal R(V_{n-1}')\rightarrow\mathcal R(\mg{1,n})\rightarrow\mathcal R(V_n)\rightarrow0
$$
$$
\mathbb Z\rightarrow\mathcal R(\mg{1,n})\rightarrow\mathbb Z\rightarrow 0.
$$
Since $\mathcal R(V_n)$ is $\mathbb Z$ and we are patching in the divisor $\im\tau_{n-1}\cong V_{n-1}'$,
whose integral tautological ring is $\mathbb Z$, we see that the integral tautological ring
of $\mg{1,n}$ is concentrated in degrees 0 and 1. Therefore
$\mathcal R(\mg{1,n})=\mathbb Z[\lambda_1]/(12\lambda_1,\lambda_1^2)$.
\end{proof}

\begin{theorem}
For $5\leq n\leq10$, the integral Chow ring of $\mg{1,n}$ is
$$
\CH(\mg{1,n})=\frac{\mathbb Z[\lambda_1]}{(12\lambda_1,\lambda_1^2)}.
$$
\end{theorem}

\begin{proof}
The Chow ring of $\mg{1,n}$ is tautological for $5\leq n\leq 10$ by Proposition \ref{mg1n is tautological} and
the tautological ring was computed in the above Proposition.
\end{proof}

\newpage

\appendix
\section{Automorphisms of marked elliptic curves}
\label{autappendix}

In this Appendix we note the following facts about automorphisms of marked elliptic curves.

\begin{proposition}
Over a field $\mathbbm k$ of characteristic not equal to 2 or 3, there exists:
\begin{itemize}
\item
one-pointed elliptic curves with automorphism groups $\mu_2,\mu_4,$ and $\mu_6$;
\item two-pointed
elliptic curves with automorphism groups $\mu_2, \mu_3$, and $\mu_4$;
\item three-pointed elliptic curves
with automorphism groups $\mu_2$ and $\mu_3$;
\item and four-pointed elliptic curves with automorphism
group $\mu_2$.
\end{itemize}
Every four-pointed elliptic curve with $\mu_2$ automorphisms has $p_2,p_3,p_4$
colinear, and every $n$-pointed elliptic curve with $n\geq5$ has no (non-trivial) automorphisms.
\end{proposition}

\begin{proof}

Recall the Weierstrass form for elliptic curves:

\begin{theorem}[Weierstrass]
Any one-pointed smooth elliptic curve over a field $\mathbbm k$ of characteristic not equal to 2 or 3
can be written in the form $y^2z=x^3+axz^2+bz^3$,
where the marked point is the point at infinity $[0:1:0]$. Moreover, if we denote such a curve by
$C_{(a,b)}$, then $C_{(a,b)}\cong C_{(a',b')}$ if and only if $(a',b')=(t^4a,t^6b)$. The isomorphism
between these curves is given by $[x:y:z]\mapsto[t^2x:t^3y:z]$. Lastly, an elliptic curve
is smooth if and only if $D=4a^3+27b^2=0$, nodal if and only if $D=0$ and $(a,b)\neq(0,0)$, and
cuspidal if and only if $(a,b)=(0,0)$.
\end{theorem}

From this we see that an elliptic curve with $n$ marked points over $\mathbbm k$ is determined by
a choice of $(a,b)$ and $p_2,\dots,p_n$, $p_i=(x_i,y_i)$, and that the automorphisms of this
curve are given by the $t\in\mathbb G_m$ such that $t\cdot(a,b)=(t^4a,t^6b)=(a,b)$
and $t\cdot p_i=(t^2x_i,t^3y_i)=(x_i,y_i)$.

Now for each $m>1$ let $\zeta_m$ denote a primitive $m^{\text{th}}$ root of unity.
From $(t^4a,t^6b)=(a,b)$ we see that the automorphism group
of every one-pointed elliptic curve contains a copy of $\mu_2$ corresponding
to $t=\zeta_2=-1$, the involution. Additionally, the curves $C_{(1,0)}$ and $C_{(0,1)}$ are fixed by
$\mu_4=\left<\zeta_4\right>$ and $\mu_6=\left<\zeta_6\right>$.
Since any automorphism of an $n$-pointed elliptic elliptic curve $(C,p_1,\dots,p_n)$ is in particular
an automorphism of $(C,p_1)$, they must all correspond to elements of $\mu_2, \mu_4$, or $\mu_6$.

The element $\zeta_2$ is an automorphism of every elliptic curve and
induces the map $\zeta_2:[x:y:z]\mapsto[x:-y:z]$, and so for
a point $p_i\neq\infty$ to be fixed by this we must have $p_i=[x:0:1]$. Then we have
$$
y^2=x^3+ax+b
$$
$$
0=x^3+ax+b,
$$
which has at most three solutions. Therefore the involution $\iota=\zeta_2$ fixes at most four points in total.
An example of a four-pointed elliptic curve with automorphism group $\mu_2$ is
$(C_{(-1,0)}, \infty, [1:0:1], [0:0:1], [-1:0:1])$. Notice that any four-pointed elliptic curve
fixed by the involution must have $p_2,p_3,p_4$ colinear, as each point lies on the line $y=0$.

The element $\zeta_4$ is an automorphism of the curve corresponding to $(1,0)$ and induces the map
$\zeta_4:[x:y:z]\mapsto[-x:\zeta_4^3y:z]$, and so for a point $p_i\neq\infty$ to be fixed by this we must have
$p_i=[0:0:1]$, which is indeed a point on the curve $C_{(1,0)}$. Therefore there is exactly one two-pointed
elliptic curve with automorphism group $\mu_4$, the curve
$(C_{(1,0)},\infty,[0:0:1])$.

The element $\zeta_6$ is an automorphism of the curve corresponding to $(0,1)$ and induces the map
$\zeta_6:[x:y:z]\mapsto[\zeta_3x:-y:z]$, and so for a point $p_i\neq\infty$ to be fixed by this we must have
$p_i=[0:0:1]$, which is \textit{not} a point on the curve $C_{(0,1)}$. Therefore there is no
$n$-pointed elliptic curve with automorphism group $\mu_6$ for $n\geq 2$.

Lastly, the element $\zeta_6^2=\zeta_3$ is an automorphism of the curve corresponding to
$(0,1)$ and induces the map $\zeta_3:[x:y:z]\mapsto[\zeta_3^2x:y:z]$, and so for a point
$p_i\neq\infty$ to be fixed by this we must have
$p_i=[0:y:z]$. Then we have
$$
y^2=x^3+ax+b
$$
$$
y^2=1.
$$
Therefore an example of a three pointed elliptic curve with automorphism group $\mu_3$ is
$(C_{(0,1)},\infty,[0:1:1],[0:-1:1])$.

This exhausts all possible automorphisms, and so there are no $n$-pointed elliptic curves with non-trivial
automorphisms for $n\geq 5$.
\end{proof}

\newpage
\bibliographystyle{alpha}
\bibliography{/Users/martinbishop/documents/research/citations/bibliography.bib}

\end{document}